\documentclass[10pt]{amsart}
\usepackage{amsmath}
\usepackage{amssymb}
\usepackage{enumerate}
\usepackage{amsbsy}
\usepackage{amsfonts}
\usepackage{color}

\newtheorem{thm}{Theorem}[section]
\newtheorem{lem}[thm]{Lemma}
\newtheorem{prop}[thm]{Proposition}

\numberwithin{equation}{section}

\newcommand{\rt}{\mathbb{R}^2}

\newcommand{\les}{\lesssim}
\newcommand{\lam}{\lambda}

\newcommand{\vp}{\varphi}

\newcommand{\de}{\delta}
\newcommand{\al}{\alpha}

\newcommand{\epm}{{e^{\pm it\left<D\right>}}}
\newcommand{\emp}{{e^{\mp it\left<D\right>}}}
\newcommand{\pkm}{{P_k^1}}
\newcommand{\brad}{\left<D\right>}

\begin{document}

	\title[Dirac equations]{Small data scattering of Dirac equations with Yukawa type potentials in $L_x^2(\mathbb R^2)$}

	\author[Y. Cho]{Yonggeun Cho}
	\address{Department of Mathematics, and Institute of Pure and Applied Mathematics, Jeonbuk National University, Jeonju 54896, Republic of Korea}
	\email{changocho@jbnu.ac.kr}

	\author[K. Lee]{Kiyeon Lee}
	\address{Department of Mathematics, Jeonbuk National University, Jeonju 54896, Republic of Korea}
	\email{leeky@jbnu.ac.kr}

	\thanks{2010 {\it Mathematics Subject Classification.} 35Q55, 35Q40.}
	\thanks{{\it Keywords and phrases.} Dirac equations, Yukawa type potential, global well-posedness, small data scattering, null structure, $U^p-V^p$ space.}
	
	\begin{abstract}
 We revisit the Cauchy problem of nonlinear massive Dirac equation with Yukawa type potentials $\mathcal F^{-1}\left[(b^2 + |\xi|^2)^{-1}\right]$ in 2 dimensions. The authors of \cite{tes2d, geosha} obtained small data scattering and large data global well-posedness in $H^s$ for $s > 0$, respectively. In this paper we show that the small data scattering occurs in $L_x^2(\mathbb R^2)$. This can be done by combining bilinear estimates and modulation estimates of \cite{yang, tes2d}. 
	\end{abstract}

		\maketitle

\section{Introduction}
We consider the following Cauchy problem for an nonlinear Dirac Hartree-type equation:
\begin{eqnarray}\label{maineq}
	\left\{\begin{array}{l}
	(-i\partial_t + \al\cdot D + m\beta) \psi =  (V* \left<\psi,\beta\psi\right>)\beta\psi \;\; \mathrm{in}\;\; \mathbb{R}^{1+2}\\
	\psi(0) = \psi_0 \in L_x^2(\rt),
	\end{array} \right.
\end{eqnarray}
where $D=-i\cdot\nabla$, $\left< \cdot,\cdot\right> = \left< \cdot,\cdot\right>_{\mathbb{C}^{2}}$,  and $\psi:\mathbb{R}^{1+2} \to \mathbb{C}^2$ is the spinor field represented by a column vector. We define the Dirac matrices $\alpha, \beta$ by dimensions as follows:

\begin{align*}
\al = (\al^1, \al^2), \quad\al^1 = \left(\begin{array}{ll} 0 & \;1 \\ 1 & \;0 \end{array} \right), \quad \al^2 = \left(\begin{array}{ll} 0 & \;-i \\ i & \;\;\;0 \end{array} \right), \quad \beta = \left(\begin{array}{ll} 1 & \;\;\;0 \\ 0 & \;-1 \end{array} \right).
\end{align*}

The constant $m > 0$ is a physical mass parameter and the symbol $*$ denotes convolution in $\rt$ and the potential $V$ is defined by $\mathcal F^{-1}[(b^2 + |\xi|^2)^{-1}] $ for some fixed constant $ b > 0$. More explicitly, for a constant $a > 0$,   $$V(x) = a\int_0^\infty e^{-b^2r- |x|^2/4r}\frac{dr}{r} \sim \left\{\begin{array}{ll}
e^{-b|x|}|bx|^{-\frac12}& |x| \gtrsim 1,\\ - \ln|x| & |x| \ll 1.
\end{array}\right.$$

The equation \eqref{maineq} with Yukawa potential was derived by uncoupling the Dirac-Klein-Gordon system in $\mathbb{R}^{1+2}$:
\begin{eqnarray}\label{DKG}
\left\{\begin{array}{l}
(-i\partial_t + \al\cdot D + m\beta) \psi =  \phi\beta\psi,\\
(\partial_t^2 - \Delta +M^2)\phi = \left<\psi,\beta\psi\right>.
\end{array} \right.
\end{eqnarray}
Let us assume that the scalar field $\phi$ is a standing wave of the form $\phi(t,x) = e^{i\lam t}\rho(x)$. Then the Klein-Gordon part of \eqref{DKG} becomes
$$
(-\Delta -\lam^2 + M^2)\phi = \left<\psi,\beta\psi\right>.
$$
If $b^2 := M^2-\lambda^2 > 0$, then we get the equation \eqref{maineq}.

The equation \eqref{maineq} obeys mass conservation law. If a solution $\psi$ is sufficiently smooth, then the mass $\|\psi(t)\|_{L_x^2}^2$ is conserved, that is, $\|\psi(t)\|_{L_x^2}^2 = \|\psi_0\|_{L_x^2}^2$ for all $t$ within an existence time interval. See \cite{geosha}.

Now let us consider a scaled function $\widetilde \psi$ defined by $\widetilde \psi(t, x) = m^{-\frac32}\psi\left(\frac tm, \frac xm \right)$. Then by a direct calculation $\widetilde \psi$ satisfies the equation:
$(-i\partial_t + \al\cdot D + \beta) \widetilde \psi =  (\widetilde V* \langle \psi,\beta\widetilde\psi\rangle)\beta\widetilde\psi$, where $\widetilde V = \mathcal F^{-1}[(\frac{b^2}{m^2} + |\xi|^2)^{-1}]$. Since the changed potential is essentially the same type as $V$ up to constant, for the Cauchy problem \eqref{maineq} we assume that $m = 1$ in this paper. 

We use the representation of solution based on the massive Klein-Gordon equation. For this purpose, let us define the energy projection operators $\Pi_\pm(D)$ by
$$
\Pi_{\pm}(D) := \frac12 \left( I \pm \frac1{\brad}[\alpha \cdot D + \beta] \right),
$$
where $\left< D\right> := \mathcal F^{-1}\langle \xi \rangle\mathcal F$ and $\langle \xi \rangle := (1 + |\xi|^2)^\frac12$ for any $\xi \in \mathbb R^2$.
Then we get
\begin{align}
\al\cdot D + \beta = \brad (\Pi_{+}(D) - \Pi_{-}(D)),
\end{align}
and
\begin{align}\label{proj-commu}
\Pi_{\pm}(D)\Pi_{\pm}(D) = \Pi_{\pm}(D), \;\; \Pi_{\pm}(D)\Pi_{\mp}(D) = 0.
\end{align}
We denote $\Pi_\pm(D)\psi $ by $\psi_\pm$. Then the equation \eqref{maineq} becomes the following system of semi-relativistic Hartree equations:
\begin{align}\label{maineq-decom}
 (-i\partial_t \pm \brad)\psi_\pm = \Pi_{\pm}(D)[(V* \left<\psi,\beta\psi\right>)\beta\psi]
\end{align}
with initial data $\psi_{\pm}(0,\cdot) = \psi_{0, \pm} := \Pi_{\pm}(D)\psi_0$. The free solutions of \eqref{maineq-decom} are $e^{\mp it\brad}\psi_{0, \pm}$, respectively, where
$$
e^{\mp it\langle D\rangle}f(x) = \mathcal F^{-1}( e^{\mp it \langle \xi \rangle}\mathcal F f) = \frac1{(2\pi)^2}\int_{\mathbb R^2} e^{i(x\cdot \xi \mp t\langle \xi \rangle)}\widehat f(\xi)\,d\xi.
$$
Here $\mathcal F, \mathcal F^{-1}$ are Fourier transform, its inverse, respectively.
Then by Duhamel's principle the Cauchy problem \eqref{maineq-decom} is equivalent to solving the integral equations:
\begin{align}\label{inteq0}
\psi_\pm(t) = e^{\mp it\langle D\rangle}\psi_{0, \pm} + i\int_0^t e^{\mp i(t-t')\langle D\rangle}\Pi_{\pm}(D)[(V*\langle \psi, \beta \psi\rangle) \beta \psi](t')\,dt'.
\end{align}
We call that the solution $\psi$ scatters forward (or backward) in $H^s$ if there exist $\psi^\ell \in C(\mathbb R; H^s)$,  linear solutions to $(-i\partial_t + \alpha \cdot D + \beta)\psi = 0$, such that
\begin{align}\label{sense1}
\|\psi(t) - \psi^\ell(t)\|_{H^s} \to 0\;\;\mbox{as}\;\;t \to +\infty\; (-\infty,\;\;\mbox{respectively}).
\end{align}
Equivalently, $\psi$ is said to scatter forward ( or backward) in $H^s$ if there exist $\psi_{\pm }^\ell := e^{\mp it\brad}\varphi_{\pm } \;(\varphi_{\pm} \in  H^s)$ such that
\begin{align}\label{sense2}
\|\psi_\pm(t) - \psi_{\pm}^\ell(t)\|_{H^s} \xrightarrow {t \to \pm \infty} 0.
\end{align}



Recently,  Yang \cite{yang} and Tesfahun \cite{tes3d} showed, independently, small data scattering results on $H^s(\mathbb{R}^3)$ for $s > 0$ in $3$ dimensions. They developed the bilinear methods based on the null structure and $U^p-V^p$ space. At the same time, Tesfahun \cite{tes2d} considered 2d problem \eqref{maineq} and obtained the scattering in $H^s(\mathbb R^2) (s > 0)$. In \cite{geosha} the global well-posedness was shown in $H^s(\mathbb{R}^2)$ for $s>0$ without the smallness of initial data. In \cite{grupec} the authors considered the global well-posedness of 2d Dirac-Klein-Gordon system with data in $L_x^2 \times H^\frac12\times H^{-\frac12}$. There has not been known about the global well-posedness and scattering in $L_x^2$ of the single equation \eqref{maineq}. In this paper, we tackle the scattering problem in $L_x^2(\mathbb{R}^2)$  and obtain the following theorem.
\begin{thm}\label{mainthm}
If $\|\psi_0\|_{L_x^2}$ is sufficiently small, then there exists a unique global solution $\psi \in C(\mathbb{R};L_x^2)$ to \eqref{maineq}, which scatters in $L_x^2$.
\end{thm}
We show Theorem \ref{mainthm} by adopting the bilinear method of Yang and Tesfahun. Tesfahun's method relies on the logarithmic interpolation between $U^p$ spaces, which results in $\varepsilon$-regularity loss for the high-high-low interaction part. To overcome it we use Yang's bilinear estimates on the $V^2$ space and fast decay in frequency of 2d Yukawa potential. Unfortunately, our method cannot be applied to 3d problem directly because the bilinear estimate is not strong enough to remove the $\varepsilon$-regularity loss. The 3d scattering problem remains still open in $L_x^2$ and will be treated as a future.

\noindent{\bf Notations.}
\\

\noindent $(1)$ $\|\cdot\|$ denotes $\|\cdot\|_{L_{t,x}^2}$.\\

\noindent$(2)$ (Mixed-normed spaces) For a Banach space $X$ and an interval $I$, $u \in L_I^q X$ iff $u(t) \in X$ for a.e.$t \in I$ and $\|u\|_{L_I^qX} := \|\|u(t)\|_X\|_{L_I^q} < \infty$. Especially, we denote  $L_I^qL_x^r = L_t^q(I; L_x^r(\rt))$, $L_{I, x}^q = L_I^qL_x^q$, $L_t^qL_x^r = L_{\mathbb R}^qL_x^r$.\\

\noindent$(3)$ (Littlewood-Paley operators) Let $\rho$ be a Littlewood-Paley function such that $\rho \in C^\infty_0(B(0, 2))$ with $\rho(\xi) = 1$  for $|\xi|\le 1$ and define $\rho_{k}(\xi):= \rho\left(\frac {\xi}{2^k}\right) - \rho\left(\frac{\xi}{2^{k-1}}\right)$ for $k \in \mathbb{Z}$. Then we define the frequency projection $P_k$ by $\mathcal{F}(P_k f)(\xi) = \rho_{k}(\xi)\widehat{f}(\xi)$, and also $P_{\le k} := I - \sum_{k' > k}P_{k'}$. In addition $P_{k_1 \le \cdot \le k_2} := \sum_{k_1 \le k \le  k_2}P_k$. For $ k \in \mathbb{Z}$ we denote $\widetilde{\rho_k} = \rho_{k-1} + \rho_k + \rho_{k+1}$. In particular, $\widetilde{P_k}P_k = P_k\widetilde{P_k} = P_k$, where $\widetilde{P_k} = \mathcal{F}^{-1}\widetilde{\rho_k}\mathcal{F}$. Next we define a Fourier localization operators $P_k^1$ as follow:
$$
P_k^1 f =  \left\{ \begin{array}{ll} 0  &\;\;\;\mbox{if}\;\; k < 0 ,\\ P_{\le 0}f &\;\;\;\mbox{if}\;\; k = 0 ,\\ P_k f & \;\;\;\mbox{if}\;\; k > 0.   \end{array} \right .
$$
Especially, we denote $P_k^1 f $ by $f_{k}$ for any measurable function $f$.\\


\noindent$(4)$ As usual different positive constants depending only on $a, b$ are denoted by the same letter $C$, if not specified. $A \lesssim B$ and $A \gtrsim B$ mean that $A \le CB$ and
$A \ge C^{-1}B$, respectively for some $C>0$. $A \sim B$ means that $A \lesssim B$ and $A \gtrsim B$.

\section{Function spaces}

We explain concisely $U^p-V^p$ spaces. For more details, we refer the readers to \cite{ haheko,haheko2, kota, kotavi}.
Let $1 \le p < \infty$ and $\mathcal{I}$ be a collection of finite partitions $\{ t_0 , \cdots, t_N\}$ satisfying $-\infty < t_0 < \cdots < t_N \le \infty$. If $t_N =\infty$, by convention, $u(t_N) := 0$ for any $u : \mathbb R \to L_x^2(\mathbb R^2)$.
Let us define a $U^p$-atom by a step function $a : \mathbb{R} \to L_x^2$ of the form
$$
a(t) = \sum_{k=1}^N \chi_{[t_{k-1},t_k)}\phi(t) \;\;\mbox{with}\;\; \sum_{k=1}^N \|\phi\|_{L_x^2}^p=1.
$$
Then the $U^p$ space is defined by
$$
U^p = \left\{ u = \sum_{j=1}^\infty \lam_j a_j : \mbox{$a_j$ are $U^p$-atoms and $\{\lam_j\} \in \ell^1$ },  \|u\|_{U^p}< \infty \right\},
$$
where the $U^p$-norm is defined by
$$
\|u\|_{U^p}:= \inf_{\mbox{representation of $u$}} \;\;\sum_{j=1}^\infty|\lam_j|.
$$
We next define $V^p$ as the space of all right-continuous functions $v : \mathbb{R} \to L_x^2$ satisfying that $\underset{t \to -\infty}{\lim}v(t) =0$ and the norm
$$
\|v\|_{V^p} := \sup_{\{t_k\} \in \mathcal{I}} \left( \sum_{k=1}^N \|v(t_k)  - v(t_{k-1})\|_{L_x^2}^p \right)^{\frac1p}
$$
is finite.

We introduce several key properties of $U^p$ and $V^p$ spaces.
\begin{lem}[\cite{haheko}]\label{embedd} Let $1 \le p < q <\infty$. Then the following holds.
	\item[(i)] $U^p$ and $V^p$ are Banach spaces.
	\item[(ii)] The embeddings $U^p \hookrightarrow V^p \hookrightarrow U^q \hookrightarrow L^{\infty}(\mathbb{R};L_x^2) $ are continuous.
\end{lem}
These spaces have the useful duality property.
\begin{lem}[Corollary of \cite{kotavi}]\label{duality} Let $u \in U^p$ be absolutely continuous with $1 < p < \infty$. Then
	$$
	\|u\|_{U^p} = \sup \left\{ \int \left<u', v \right>_{L_x^2} dt : v \in C_0^{\infty}, \;\; \|v\|_{V^{p'}}=1 \right\}.
	$$
\end{lem}

Now let us define the adapted function spaces $U_{\pm}^p,\; V_{\pm}^p$ as follows:
$$
\|u\|_{U_{\pm}^p} : =\|\epm u\|_{U^p} \;\;\mbox{and}\;\;  \|u\|_{V_{\pm}^p} : =\|\epm u\|_{V^p}.
$$

\begin{prop}[Transfer principle, Proposition 2.19 of \cite{haheko}]\label{transfer} Let
	$$
	T : L_x^2 \times L_x^2 \times \cdots \times L_x^2 \to L_{loc}^1
	$$
	be a multilinear operator. If
	$$
	\left\|T\left( e^{\pm_1 it\brad}f_1 , e^{\pm_2 it\brad}f_2 , \cdots ,e^{\pm_k it\brad}f_k \right)\right\|_{L_t^qL_x^r} \les \prod_{j=1}^{k}\|f_j\|_{L_x^2}
	$$
	for some $1 \le q, r \le \infty$ and $\pm_j \in \{\pm\}$, then we have
	$$
	\|T(u_1, u_2,\cdots, u_k)\|_{L_t^q L_x^r} \les \prod_{j=1}^k \|u_j\|_{U_{\pm_j }^q}.
	$$
\end{prop}

\section{Bilinear estimates}

In this section, we list basic bilinear estimates based on the estimates of \cite{tes2d, tes3d, yang}.

\begin{lem}\label{ltwo-esti}
	Let $k_j \in \mathbb{Z}$, $\psi_j \in V_{\pm_j}^2\;(j = 1, 2)$, and   $\Pi_{\pm_j}(D)P_{k_j}^1\psi_j = \psi_j$.  Then
\begin{align*}
\left\|\left< \psi_1 , \beta \psi_2 \right> \right\| \les 2^{p k_1 + (1- p)k_2}\|\psi_1\|_{V_{\pm_1}^2}\|\psi_2\|_{V_{\pm_2}^2}
		\end{align*}
for any $0 <  p < 1$.
\end{lem}

\begin{proof}[Proof of Lemma \ref{ltwo-esti}] For the proof we use the well-known Strichartz estimates (for instance see\cite{choz, chozxi}):
Suppose $(q,r)$ satisfies that $2 \le r < \infty$ and $\frac1q = \frac12 -\frac1r$. Then
\begin{align}\label{KG-stri}
\|\epm \pkm f\|_{L_t^q L_x^r} \les \left< 2^k\right>^{\frac2q}\|P_k^1 f\|_{L_x^2}.
\end{align}		

From  \eqref{KG-stri}, Proposition \ref{transfer}, and Lemma \ref{embedd}   we get
	\begin{align}\label{stri-trans}
	\| \pkm \psi\|_{L_t^q L_x^r} \les 2^{\frac{2k}q}\|\psi\|_{U_{\pm}^q} \les 2^{\frac{2k}q}\|\psi\|_{V_{\pm}^2}
	\end{align}
	for $2 < q < \infty$. Hence, by \eqref{stri-trans} we get
	\begin{align*}
	\left\|\left<\psi_1, \beta \psi_2\right> \right\| &\les \|\psi_1\|_{L_t^q L_x^r} \|\psi_2\|_{L_t^r L_x^q} \les 2^{\frac2q k_1  + \left( 1- \frac2q\right)k_2 } \|\psi_1\|_{U_{\pm_1}^q} \|\psi_2\|_{U_{\pm_2}^r}\\
	&\les 2^{\frac2q k_1  + \left( 1- \frac2q\right)k_2 }\|\psi_1\|_{V_{\pm_1}^2} \|\psi_2\|_{V_{\pm_2}^2}.
	\end{align*}
By setting $p = \frac2q$ the proof finishes.
\end{proof}

The following proposition is key estimate to be used in high-high-low interaction.
\begin{prop}[Proposition 3.6, 3.7 of \cite{yang} and  Proposition 3.7, 3.9 of \cite{choleeoz}]\label{prop-bilinear}
	Let $ \Pi_{\pm_j}(D) P_{k_j} \psi_j = \psi_{j}\in V_{\pm_j}^2$.	Assume that $ k_1, k_2 \ge 0, \,k \in \mathbb{Z} $ and  that $2^k \ll 2^{k_1} \sim 2^{k_2}$.  Then we get the following:
	\begin{enumerate}
		\item[$(i)$] If $\pm_1 = \pm_2$, $\|P_k \langle \psi_1 , \beta \psi_2 \rangle \| \les 2^{k - \frac{k_1}2} \|\psi_1\|_{V_{\pm_1}^2} \|\psi_2\|_{V_{\pm_2}^2}.$
		\item[$(ii)$] If $\pm_1 \neq \pm_2$, $\|P_k \langle \psi_1 , \beta \psi_2 \rangle \| \les 2^{\frac k2} \|\psi_1\|_{V_{\pm_1}^2} \|\psi_2\|_{V_{\pm_2}^2}.$
	\end{enumerate}
\end{prop}



\section{Proof of Theorem \ref{mainthm}}
\newcommand{\x}{X_{\pm}}

We prove Theorem \ref{mainthm} by contraction argument.
Let us define Banach spaces $\x$ and $X_{\pm, p}$ by
$$
\x := \left\{ \phi \in C(\mathbb{R} ; L_x^2) : \|\phi\|_{\x} := \left( \sum_{k \in \mathbb{Z}} \|\pkm \phi\|_{U_{\pm}^2}^2\right)^{\frac12} < \infty \right\}
$$
and $X_{\pm, p} = \{\psi = \chi_{[0,\infty)}(t)\phi : \phi \in \x \}$, respectively.
Then by the decomposition $\psi = \psi_+ + \psi_-$, where $\psi_{\pm} = \Pi_{\pm}(D)\psi$, we define a complete metric space $X_p(\de)$ as
$$
X_p(\de) := \left\{ \psi \in X_{\pm, p} : \|\psi\|_{X} := \|\psi_+\|_{X_+} + \|\psi_-\|_{X_-} \le \de \right\}
$$
with metric $\mathbf d(\psi,\phi) := \|\psi - \phi\|_{X}$ and a map $\mathcal{N}$ defined by
\begin{align*}
\mathcal{N}(\psi) = \sum_{\pm} \left[\chi_{[0,\infty)}(t)e^{\mp it\brad}\Pi_{\pm}(D)\psi_0 + i \sum_{\pm_j, j = 1,2,3} N_{\pm}(\psi_{\pm_1},\psi_{\pm_2},\psi_{\pm_3})(t)\right],
\end{align*}
where
\begin{align*}
N_{\pm}(\psi_{1},\psi_{2},\psi_{3})(t) = \int_0^t e^{\mp i(t-t')\brad}\Pi_{\pm}(D)[(V*\left< \psi_{1}, \beta\psi_{2}\right>)\beta\psi_{3}] d\,t'.
\end{align*}
Here $\sum_{\pm} A_{\pm}$ means that $A_+ + A_-$.


The linear part of $\mathcal N(\psi)$ can be estimated as follows:
\begin{align}\label{linear-est}
\left\|\chi_{[0,\infty)}\emp \Pi_{\pm}(D) \psi_0 \right\|_{\x}^2 = \sum_{k\in \mathbb{Z}} 2^{2sk} \left\|\chi_{[0,\infty)} \pkm \Pi_{\pm}(D) \psi_0 \right\|_{U_{\pm}^2}^2 \sim \|\psi_0\|_{H^{s}}^2.
\end{align}
For the nonlinear parts for $N_{\pm}(\psi)(t)$ we prove
\begin{prop}\label{nonlinear}
If $\psi_j \in X_{\pm_j, p}$, then we have
\begin{align*}
\|N_{\pm}(\psi_{1, \pm_1},\psi_{2, \pm_2},\psi_{3, \pm_3})\|_{\x} \les \prod_{j=1}^3\|\psi_{j, \pm_j}\|_{X_{\pm_j}}.
\end{align*}
\end{prop}
\noindent The proof of Proposition \ref{nonlinear} is placed in the next section. 

If $\de$ is small enough that $C\delta^3 \le \frac\delta8$ and $\psi_0$ satisfies $C\|\psi_0\|_{L_x^2} \le \frac\delta2$, Proposition \ref{nonlinear} together with linear estimate \eqref{linear-est} leads us to
\begin{align*}
\|\mathcal{N}(\psi)\|_{X} = \sum_{\pm}\|\Pi_{\pm}\mathcal N(\psi)\|_{X_{\pm}} \le C(\|\psi_0\|_{L_x^2} + \|\psi\|_X^3) \le \de
\end{align*}
where $\|\phi\|_X:= \|\phi_+\|_{X_+} + \|\phi_-\|_{X_-}$.
This yields that $\mathcal N$ is a self-mapping on $X_p(\delta)$. In particular, we get
\begin{align*}
\mathbf d \Big(\mathcal{N}(\psi), \mathcal{N}(\phi)\Big) = \|\mathcal{N}(\psi) - \mathcal{N}(\phi)\|_{X} &\le C\left(\|\psi\|_{X}+ \|\phi\|_{X}\right)^2 \|\psi- \phi\|_{X} \le 4C\de^2\|\psi- \phi\|_{X} \le \frac12\mathbf d(\psi, \phi).
\end{align*}
Hence $\mathcal N: X_p(\delta) \to X_p(\delta)$ is a contraction mapping for sufficiently small $\de$ and then we get a unique solution $\psi_p \in L^\infty([0, \infty); L_x^2)$ to \eqref{maineq}. The time continuity and continuous dependency on data follow readily from the formula $\psi_p = \mathcal N(\psi_p)$ and Proposition \ref{nonlinear}. By the time symmetry of \eqref{maineq} we also obtain a unique solution $\psi_n \in C((-\infty, 0], L_x^2)$ with the continuous dependency on data. Defining $\psi = \psi_p + \psi_n$, we get the global well-posedness of \eqref{maineq}.

Now we move onto the scattering property of \eqref{maineq-decom}. Since the backward scattering can be treated similarly to the forward one, we omit its proof. For $k \ge 0$ let us define
$$
\vp_\pm:= \epm P_k^1 \mathcal{N}_{\pm}(\psi),
$$
where $\mathcal{N}_{\pm}(\psi) = \lim_{t\to \infty}\sum_{\pm_j} N_{\pm}(\psi_{\pm_1},\psi_{\pm_2},\psi_{\pm_3})(t)$. Then Lemma \ref{embedd} shows that
$$
\vp_\pm \in V_{\pm}^2.
$$
Since $\sum_{k \ge 0}  \|\vp_\pm\|_{V_\pm^2} \les 1$, we have
$$
\phi_\pm :=\lim_{t \to \infty}\vp_\pm \in L_x^2
$$
and
$$
\|\psi_\pm(t) - \emp \phi_\pm\|_{L_x^2}  \xrightarrow{t \to \infty} 0.
$$
This completes the proof of scattering part.

\section{Proof of Proposition \ref{nonlinear}}
By duality we obtain
\begin{align*}
 & \left\|P_{k_4}^1 \int_0^t e^{\mp i(t-t')\brad}\Pi_{\pm}(D)[(V*\left< \psi_{1,\pm_1}, \beta\psi_{2,\pm_2}\right>)\beta\psi_{3,\pm_3}] dt' \right\|_{U_{\pm}^2}\\
 & \qquad\qquad =\left\|P_{k_4}^1 \int_0^t e^{\pm it'\brad}\Pi_{\pm}(D)[(V*\left< \psi_{1,\pm_1}, \beta\psi_{2,\pm_2}\right>)\beta\psi_{3,\pm_3}] dt' \right\|_{U^2}\\
 & \qquad\qquad = \sup_{\begin{subarray}{ll} \|\phi\|_{V^2}=1 \\ \;\;\phi \in C_0^\infty  \end{subarray}}\left| \iint (V*\left< \psi_{1,\pm_1}, \beta\psi_{2,\pm_2}\right>)\left<\beta\psi_{3,\pm_3}, \Pi_{\pm}(D)P_{k_4}^1 e^{\mp it\brad}\phi\right> dt dx\right|\\
& \qquad\qquad = \sup_{\|\psi_4\|_{V_{\pm_4}^2}=1}\left| \iint (V*\left< \psi_{1,\pm_1}, \beta\psi_{2,\pm_2}\right>)\left<\beta\psi_{3,\pm_3}, P_{k_4}^1 \psi_{4,\pm_4}\right> dt dx\right|.
\end{align*}
Then by dyadic decomposition we have
\begin{align*}
&\|N_{\pm_4}(\psi_{1, \pm_1},\psi_{2, \pm_2},\psi_{3, \pm_3})\|_{X_{\pm_4}}^2\\
&\qquad= \sum_{k_4 \in \mathbb{Z}}  \| P_{k_4}^1 N_{\pm_4}(\psi_{1, \pm_1},\psi_{2, \pm_2},\psi_{3, \pm_3})\|_{U_{\pm_4}^2}^2\\
&\qquad \les  \sum_{k_4 \in \mathbb{Z}}\left( \sup_{\|\psi_4\|_{V_{\pm_4}^2}=1} \sum_{k,k_1,k_2,k_3 \in \mathbb{Z}}  \left| \iint P_k \left(V*\left< \psi_{1, \pm_1, k_1}, \beta\psi_{2, \pm_2, k_2}\right>\right) \widetilde{P_k}(\left<\beta\psi_{3, \pm_3, k_3},  \psi_{4, \pm_4, k_4}\right>) dt dx\right| \right)^{2}\\
&\qquad \les \sum_{k_4\in\mathbb{Z}} \left( \sup_{\|\psi_4\|_{V_{\pm_4}^2}=1}( I_1 + I_2 + I_3) \right)^2,
\end{align*}
where $\psi_{j, \pm_j, k_j} = P_{k_j}^1\Pi_{\pm_j}(D)\psi_j$ and
\begin{align*}
I_1 = \sum_{\begin{subarray}{c} k_1, k_2 \in \mathbb Z\\ 2^k \ll 2^{k_3} \sim 2^{k_4}\end{subarray}}|\cdots|,\quad I_2 = \sum_{\begin{subarray}{c} k_1, k_2 \in \mathbb Z\\ 2^{k_4} \sim 2^k  \gg 2^{k_3}\end{subarray}}|\cdots|,\quad I_3 = \sum_{\begin{subarray}{c} k_1, k_2 \in \mathbb Z \\ 2^k \sim 2^{k_3} \gtrsim 2^{k_4}\end{subarray}}|\cdots|.
\end{align*}
We subdivide $I_j$ as follows:
$$
I_1 =  I_{11} + I_{12} + I_{13} := \sum_{\begin{subarray}{c} 2^k \lesssim 2^{k_1} \sim 2^{k_2}\\ 2^k \ll 2^{k_3} \sim 2^{k_4}\end{subarray}}|\cdots| + \sum_{\begin{subarray}{c} 2^{k_2} \ll  2^{k_1} \sim 2^k   \\ 2^k \ll 2^{k_3}\sim 2^{k_4}\end{subarray}}|\cdots| + \sum_{\begin{subarray}{c} 2^{k_1} \ll 2^{k_2} \sim 2^k \\ 2^k \ll 2^{k_3}\sim 2^{k_4}\end{subarray}}|\cdots|,
$$

$$
I_2 = I_{21} + I_{22} + I_{23} := \sum_{\begin{subarray}{c} 2^k \lesssim 2^{k_1} \sim 2^{k_2}\\ 2^{k_4} \sim 2^k  \gg 2^{k_3}\end{subarray}}|\cdots| + \sum_{\begin{subarray}{c} 2^k \sim 2^{k_1} \gg 2^{k_2} \\ 2^{k_4} \sim 2^k  \gg 2^{k_3}\end{subarray}}|\cdots|  + \sum_{\begin{subarray}{c} 2^{k_1} \ll 2^{k_2} \sim 2^k \\ 2^{k_4} \sim 2^k  \gg 2^{k_3}\end{subarray}}|\cdots|,
$$

$$
I_3 = I_{31} + I_{32} + I_{33} := \sum_{\begin{subarray}{c} 2^k \lesssim 2^{k_1} \sim 2^{k_2}\\ 2^k \sim 2^{k_3} \gtrsim 2^{k_4}\end{subarray}}|\cdots| + \sum_{\begin{subarray}{c} 2^k \sim 2^{k_1} \gg 2^{k_2} \\ 2^k \sim 2^{k_3} \gtrsim 2^{k_4}\end{subarray}}|\cdots|  + \sum_{\begin{subarray}{c} 2^{k_1} \ll 2^{k_2} \sim 2^k \\ 2^k \sim 2^{k_3} \gtrsim 2^{k_4}\end{subarray}}|\cdots|.
$$
It suffices to show that for each $I_{ij} \;(i,j = 1, 2, 3)$
\begin{align}\label{est-ij}
\mathcal I_{ij} := \sum_{k_4\in\mathbb{Z}} \left( \sup_{\|\psi_4\|_{V_{\pm_4}^2}=1}[I_{ij}]^2\right) \lesssim \prod_{j = 1}^3\|\psi_{j, \pm_j}\|_{X_{\pm_j}}^2.
\end{align}

In fact, they can be handled as follows.
%
%
%
\newcommand{\xx}{X_{\pm}}
 By Proposition \ref{prop-bilinear} we have
\begin{align*}
\mathcal I_{11} &\les \sum_{k_4 \in \mathbb{Z}} \left( \sum_{\begin{subarray}{ll} 2^k \les 2^{k_1}\sim 2^{k_2}\\ 2^k \ll 2^{k_3}\sim 2^{k_4} \end{subarray}} \left<2^k\right>^{ -2}    \left\| P_k\left< \psi_{1,k_1}, \beta\psi_{2,k_2}\right>\right\|   \left\| \widetilde{P_k}\left<\beta\psi_{3,k_3}, \Pi_{\pm}(D)P_{k}^1 \psi_4\right>\right\|\right)^2 \\
&\les \sum_{k_4 \in \mathbb{Z}} \left( \sum_{\begin{subarray}{ll} 2^k \les 2^{k_1}\sim 2^{k_2}\\ 2^k \ll 2^{k_3}\sim 2^{k_4} \end{subarray}} 2^{k} \left<2^k\right>^{ -2}   \|\psi_{1,k_1}\|_{V_{\pm_1}^2}\|\psi_{2,k_2}\|_{V_{\pm_2}^2}\|\psi_{3,k_3}\|_{V_{\pm_3}^2}  \right)^2\\
&\les \|\psi_1\|_{X_{\pm_1}}^2\|\psi_2\|_{X_{\pm_2}}^2  \sum_{k_4 \in \mathbb{Z}} \|\psi_{3,k_4}\|_{V_{\pm_3}^2}^2  \left( \sum_{k \in \mathbb{Z}}  2^{ k}\left<2^k\right>^{ -2}  \right)^2\\
&\les \prod_{j=1}^3 \|\psi_j\|_{X_{\pm_j}}^2.
\end{align*}

Using Lemma \ref{ltwo-esti} and Proposition  \ref{prop-bilinear},
\begin{align*}
\mathcal I_{12} &\les  \sum_{k_4 \in \mathbb{Z}} \left( \sum_{\begin{subarray}{ll} 2^{k_2} \ll 2^{k_1}\sim 2^{k}\\ 2^{k} \ll 2^{k_3}\sim 2^{k_4} \end{subarray}} \left<2^k\right>^{ -2} 2^{\frac{k_1 +k_2}2} \|\psi_{1,k_1}\|_{V_{\pm_1}^2}\|\psi_{2,k_2}\|_{V_{\pm_2}^2}2^{\frac k2}\|\psi_{3,k_3}\|_{V_{\pm_3}^2}  \right)^2\\
&\les   \sum_{k_4 \in \mathbb{Z}} \left( \sum_{\begin{subarray}{ll}  2^{k_2}\ll 2^{k_1}\\ 2^{k_3} \sim 2^{k_4} \end{subarray}} 2^{\frac32 k_1}\left<2^{k_1}\right>^{ -2} 2^{\frac12 \left(k_2 - k_1 \right)} \|\psi_{1,k_1}\|_{V_{\pm_1}^2}\|\psi_{2,k_2}\|_{V_{\pm_2}^2}\|\psi_{3,k_3}\|_{V_{\pm_3}^2}  \right)^2\\
&\les  \|\psi_3\|_{X_{\pm_3}}^2 \left( \sum_{ 2^{k_2} \ll 2^{k_1} } 2^{\frac32 k_1}\left<2^{k_1}\right>^{ -2} 2^{\frac12 \left(k_2 - k_1 \right)} \|\psi_{1,k_1}\|_{V_{\pm_1}^2}\|\psi_{2,k_2}\|_{V_{\pm_2}^2}  \right)^2\\
&\les \prod_{j=1}^3 \|\psi_j\|_{X_{\pm_j}}^2.
\end{align*}
$\mathcal I_{13}$ and $\mathcal I_{21}$ can be handled by changing the role of $\psi_1, \psi_2$, and $(\psi_1, \psi_2), (\psi_4, \psi_3)$, respectively.

As for $\mathcal I_{22}$ we apply Lemma \ref{ltwo-esti} to both $(\psi_1, \psi_2)$ and $(\psi_3, \psi_4)$ to get
\begin{align*}
\mathcal I_{22} &\les  \sum_{k_4 \in \mathbb{Z}} \left( \sum_{\begin{subarray}{ll} 2^{k_2} \ll 2^{k_1}\sim 2^{k}\\ 2^{k_3} \ll 2^{k_4}\sim 2^{k} \end{subarray}} \left<2^k\right>^{ -2} 2^{\frac{k_1 + k_2}2} \|\psi_{1,k_1}\|_{V_{\pm_1}^2}\|\psi_{2,k_2}\|_{V_{\pm_2}^2}2^{\frac{k_3+k_4}2}\|\psi_{3,k_3}\|_{V_{\pm_3}^2}  \right)^2\\
&\les   \sum_{k_4 \in \mathbb{Z}} \|\psi_{1,k_4}\|_{V_{\pm_2}^2}^2 \left( \sum_{  2^{k_2}\ll 2^{k_4}} 2^{ k_4}\left<2^{k_4}\right>^{ -1} 2^{\frac12(k_2 - k_4)} \|\psi_{2,k_2}\|_{V_{\pm_1}^2}  \right)^2\\
&\qquad \qquad \qquad \quad   \times \left( \sum_{2^{k_3} \ll 2^{k_4} } 2^{ k_4}\left<2^{k_4}\right>^{ -1} 2^{\frac12 (k_3 - k_4)} \|\psi_{3,k_3}\|_{V_{\pm_3}^2}  \right)^2\\
& \les \prod_{j=1}^3 \|\psi_j\|_{X_{\pm_j}}^2.
\end{align*}
$\mathcal I_{23}$ is treated similarly by changing the role of $\psi_1, \psi_2$.
%
%
%
%
The estimates of $\mathcal I_{3j}$ are symmetric to those of $\mathcal I_{2j}$. We have only to change the role of $\psi_3$ and $\psi_4$. This completes the proof of Theorem \ref{mainthm}.

\section*{Acknowledgements}
This work was supported by NRF-2018R1D1A3B07047782 and NRF-2021R1I1A3A04035040(Republic of Korea).



\begin{thebibliography}{00}


\bibitem{choleeoz} Y. Cho, K. Lee, and T. Ozawa, \textit{Small data scattering of 2d Hartree type Dirac equations}, preprint.

\bibitem{choz}  Y. Cho and T. Ozawa, \textit{ On the semirelativistic Hartree-type equation}, SIAM J. Math. Anal. \textbf{38} (2006), 1060--1074.



\bibitem{chozxi} Y. Cho,  T. Ozawa, and S. Xia, {\it Remarks on some dispersive estimates}, Commun. Pure Appl. Anal., \textbf{10} (2011), 1121--1128.

\bibitem{geosha} V. Georgiev and B. Shakarov, {\it Global $H^s, s>0$ large data solutions of 2D Dirac equation with Hartree type interaction}, in preprint (arXiv:2005.06853).
	
\bibitem{grupec} A. Gr\"unrock and H. Pecher, {\it Global solutions for the Dirac-Klein-Gordon system in two space dimensions}, Comm. Partial Differential Equations \textbf{35} (2010), no. 1, 89--112.


\bibitem{haheko} M. Hadac, S. Herr, and H. Koch {\it Well-posedness and scattering for the KP-II equation in a critical space }, Inst. H.Poincar\'e Anal. Non lin\'eaire, \textbf{26} (2009), 917--941.

\bibitem{haheko2}  \bysame, {\it Erratum to  "Well-posedness and scattering for the KP-II equation in a critical space" [Inst. H.Poincar\'e Anal. Non lin\'eaire, \textbf{26} (2009), 917-941] }, Inst. H.Poincar\'e Anal. Non lin\'eaire, \textbf{27} (2010), 971-972.









\bibitem{kota} H. Koch and D. Tataru, {\it A priori bounds for the 1D cubic NLS in negative Sobolev spaces}, Int. Math. Res. Not. IMRN 2007, no. 16, Art. ID rnm053, 36 pp.


\bibitem{kotavi}  H. Koch, D. Tataru, and M. Visan, Dispersive Equations and Nonlinear Waves: Generalized Korteweg–de Vries, Nonlinear Schrödinger, Wave and Schrödinger Maps, Oberwolfach Seminars 45. Basel; Birkhäuser, (2014).







\bibitem{tes2d} A. Tesfahun, {\it Long-time behavior of solutions to cubic Dirac equation with Hartree type nonlinearity in $\mathbb{R}^{1+2}$}, Int. Math. Res. Not. IMRN 2020, no. 19, 6489--6538.


\bibitem{tes3d} \bysame, {\it Small data scattering for cubic Dirac equation with Hartree type nonlinearity in $\mathbb{R}^{1+3}$}, SIAM J. Math. Anal. \textbf{52} (2020), no. 3, 2969--3003.

\bibitem{yang} C. Yang, {\it Scattering results for Dirac Hartree-type equations with small initial data}, Communications on Pure and Applied Analysis, \textbf{18 (4)} (2019), 1711-1734.
\end{thebibliography}
\end{document}